\documentclass[reqno]{amsart}
\usepackage{amsthm,amsmath,amsfonts,amssymb,enumerate,latexsym,mathrsfs,lineno}

\usepackage{graphicx}
\newtheorem{thm}{Theorem}

\newtheorem*{thm*}{Theorem}

\newtheorem*{defn*}{Definition}

\newtheorem{claimm}{Claim}

\newtheorem{lemma}{Lemma}

\newtheorem{cor}{Corollary}
\newtheorem*{cor*}{Corollary}

\newcommand{\N}{\mathbb{N}}

\newcommand{\funct}[2]{#1 \longrightarrow #2}
\newcommand{\arrows}[3]{\longrightarrow {#1}^{#2}_{#3}}
\newcommand{\restrict}[2]{#1\mspace{-2mu}\mathbin{\upharpoonright}\mspace{-1mu} #2}

\newcommand{\C}{\mathcal{C}}
\newcommand{\oC}{\mathcal{C}^<}
\newcommand{\oK}{\mathcal{K}^<}
\newcommand{\h}{\mathcal{H}}
\newcommand{\oH}{\mathcal{H}^<}

\newcommand{\oM}{\mathcal{M}^<}

\newcommand{\Aut}{\mathrm{Aut}}

\newcommand{\m}[1]{\textbf{#1}}
\newcommand{\hm}[1]{\widehat{\textbf{#1}}}

\newcommand{\mc}[1]{\widetilde{\textbf{#1}}}

\begin{document}

\author{L. Nguyen Van Th\'{e} $^*$}

\address{Institut de math\'ematiques, Universit\'e de Neuch\^atel, Rue Emile-Argand 11, 2000 Neuch\^atel, Switzerland}

\email{lionel.nguyen@unine.ch}

\thanks{$^*$ Research supported by Swiss National Fund Grant \# 20-118014.1}

\title{Some Ramsey theorems for finite $n$-colorable and $n$-chromatic graphs}

\date{July, 2009}

\begin{abstract}
Given a fixed integer $n$, we prove Ramsey-type theorems for the classes of all finite ordered $n$-colorable graphs, finite $n$-colorable graphs, finite ordered $n$-chromatic graphs, and finite $n$-chromatic graphs.  

\end{abstract}

\maketitle

\section{Introduction}

General Ramsey properties of classes of finite structures have been extensively studied over the last fourty years (see \cite{KPT} for a recent impulse), and many deep partition phenomena about finite graphs are new well understood (in the sequel, all graphs are simple and loopless, and therefore seen as structures of the form $\m{X}=(X, E^{\m{X}})$, where $X$ is a set and $E^{\m{X}}$ an irreflexive symmetric binary relation on $X$). For example, if $n$ is a fixed integer and $\m{X}, \m{Y}$ are fixed finite ordered $K_n$-free graphs (ie, graphs, with a linear ordering, not containing the complete graph $K_n$ as an induced subgraph), then for every $k \in \N$, there exists a finite ordered $K_n$-free graph $\m{Z}$ such that \[ \m{Z} \arrows{(\m{Y})}{\m{X}}{k}. \] 

That symbol means that whenever isomorphic (induced) copies of $\m{X}$ in $\m{Z}$ are colored with $k$ colors, there is an isomorphic copy $\mc{Y}$ of $\m{Y}$ in $\m{Z}$ where all copies of $\m{X}$ have same color. This result, arguably one of the most important results in structural Ramsey theory, is due to Ne\v set\v ril and R\"odl \cite{NR}. Using standard jargon, it states that the class $\oH _n$ of all finite ordered $K_n$-free graphs has the \emph{Ramsey property}, or is a \emph{Ramsey class}. Moreover, $\oH _n$ is known to be one of the few classes of finite ordered graphs with that property: Ramsey classes of finite ordered graphs have been classified be Ne\v set\v ril in \cite{N} (there is however a mistake in \cite{N}, see for example \cite{KPT}, Section 6(A) for the correct result).  

One of the reasons for which one usually works with ordered graphs instead of graphs is simply that Ramsey property fails if linear orderings are removed. However, it does not fail \emph{so} badly. For example, let $\m{X}, \m{Y}$ be fixed finite $K_n$-free graphs and let $k \in \N$. Then there exists a finite $K_n$-free graph $\m{Z}$ such that \[ \m{Z} \arrows{(\m{Y})}{\m{X}}{k, |\m{X}|!/|\Aut(\m{X})|}. \] 

That means that whenever isomorphic copies of $\m{X}$ in $\m{Z}$ are colored with $k$ colors, there is an isomorphic copy $\mc{Y}$ of $\m{Y}$ in $\m{Z}$ such that no more than $|\m{X}|!/|\Aut(\m{X})|$ colors appear on the set $\binom{\mc{Y}}{\m{X}}$ of copies of $\m{X}$ in $\mc{Y}$. Moreover, if $\m{X}$ is fixed, the number $|\m{X}|!/|\Aut(\m{X})|$ is optimal in the sense that it is minimal for the previous property to hold for every $\m{Y} \in \h _n$. Using the terminology introduced by Fouch\'e, and writing $\h _n$ for the class of all finite $K_n$-free graphs, we say that every $\m{X} \in \h _n$ has a \emph{finite Ramsey degree in} $\h _n$, equal to \[ t_{\h _n}(\m{X}) = |\m{X}|!/|\Aut(\m{X})|.\]

The purpose of this paper is to provide a full analysis of Ramsey degrees in the context of $n$-colorable and $n$-chromatic graphs, where $n$ is a fixed integer. In the ordered case, our results read as follows: we write $\oC _n$ for the class of all finite ordered $n$-colorable graphs and $\oK _n$ for the class of all finite ordered colored graphs with colors in $[n]$. Those are finite ordered graphs $(X, E^{\m{X}}, <^{\m{X}})$ together with a map $\lambda^{\m{X}}$ that colors the vertices with colors in $[n]=\{1,\ldots, n \}$ and so that not two adjacent vertices receive same color. Because elements of $\oK _n$ may be seen as finite ordered graphs where vertices receive colors in $[n]$ in such a way that certain configurations, namely pairs of adjacent vertices with same color, are forbidden, deep results contained in the aforementioned article of Ne\v set\v ril and R\"odl \cite{NR} imply that: 

\begin{thm}[Ne\v set\v ril-R\"odl, \cite{NR}]
\label{thm:RoK}
The class $\oK _n$ has the Ramsey property.
\end{thm}

In the present paper, this result is used to capture Ramsey degrees in $\oC _n$. For $\m{X}=(X, E^{\m{X}}, <^{\m{X}}) \in \oC _n$, an \emph{extension of \m{X} in $\oK _n$} is an element $\hm{X}$ of $\oK _n$ obtained from $\m{X}$ by adding a coloring $\lambda^{\m{X}}$. In that case, say also that $\m{X}$ is the \emph{reduct} of $\hm{X}$ in $\oC _n$. We denote by $\sigma(\m{X})$ the number of non isomorphic extensions of $\m{X}$ in $\oK _n$. Then: 

\begin{thm}
\label{thm:RdoC}
Every element $\m{X}$ of $\oC _n$ has a finite Ramsey degree in $\oC _n$ equal to \[ t_{\oC _n}(\m{X})=\sigma(\m{X}).\]
\end{thm}

As a direct corollary, we obtain, for the class $\chi ^< _n$ of all finite ordered $n$-chromatic graphs: 

\begin{cor}
\label{thm:Rdochi}
Every element $\m{X}$ of $\chi ^< _n$ has a finite Ramsey degree in $\chi ^< _n$ equal to \[ t_{\chi ^< _n}(\m{X})=\sigma(\m{X}).\]
\end{cor}

When linear orderings are dropped, results keep a similar flavor. We write $\C _n$ for the class of all finite $n$-colorable graphs, and $\oM _n$ for the class of all \emph{monotone} elements $\m{X}$ of $\oK _n$, ie where the coloring function $\lambda^{\m{X}}$ is increasing (when seen as a map from $(X,<^{\m{X}})$ to $[n]$). It turns out that $\oM _n$ has the Ramsey property (see Lemma \ref{thm:RoM}), and that this fact can be used to compute Ramsey degrees in $\C _n$. As previously, if $\m{X}=(X, E^{\m{X}}) \in \C _n$, an extension of $\m{X}$ in $\oM _n$ is an element $\hm{X}$ of $\oM _n$ obtained from $\m{X}$ by adding a linear ordering $<^{\m{X}}$ and an increasing coloring $\lambda^{\m{X}}$. Then, say also that $\m{X}$ is the reduct of $\hm{X}$ in $\C _n$. We denote by $\tau(\m{X})$ the number of non isomorphic extensions of $\m{X}$ in $\oM _n$. 

\begin{thm}
\label{thm:RdC}
Every element $\m{X}$ of $\C _n$ has a finite Ramsey degree in $\C _n$ equal to \[ t_{\C _n}(\m{X})=\tau(\m{X}).\]
\end{thm}

It follows that for the class $\chi _n$ of all finite $n$-chromatic graphs:

\begin{cor}
\label{thm:Rdchi}
Every element $\m{X}$ of $\chi _n$ has a finite Ramsey degree in $\chi _n$ equal to \[ t_{\chi _n}(\m{X})=\tau(\m{X}).\]
\end{cor}

None of the techniques we use is new. In fact, \emph{all} the techniques we use are standard. We feel however that the results they allow to reach deserve to be mentioned for at least two reasons. First, they extend former results of Fouch\'e \cite{F} where Theorem \ref{thm:RdC} is obtained via posets in the case $n=2$. Second, they provide new information about the partition calculus of classes of graphs which appear naturally in graph theory but which are not obviously connected to the other known classes of graphs for which the Ramsey properties are completely known.  

The paper is organized as follows: Theorem \ref{thm:RdoC} is proved in section \ref{section:RdoC}, while Theorem \ref{thm:RdC} is proved in section \ref{section:RdC}. 

\

\textbf{Acknowledgements}: I would like to express my most sincere gratitude to Alain Valette, thanks to whom this research was carried out at the Institut de Math\'ematiques de l'Universit\'e de Neuch\^atel.

\section{Proof of Theorem \ref{thm:RdoC}}

\label{section:RdoC}

Let $\m{X} \in \oC _n$. We are going to show first that $t_{\oC _n}(\m{X}) \leq \sigma(\m{X})$, and then that $t_{\oC _n}(\m{X}) \geq \sigma(\m{X})$.

\subsection{Proof of $t_{\oC _n}(\m{X}) \leq \sigma(\m{X})$}

\label{subsection:t<sigma}

The inequality $t_{\oC _n}(\m{X}) \leq \sigma(\m{X})$ is a direct consequence of Theorem \ref{thm:RoK}. Let $k \in \N$ and $\m{Y} \in \oC _n$. Let $\hm{Y}$ be any extension of $\m{Y}$ in $\oK _n$. Enumerate the extensions $\hm{X}_1 ,\ldots, \hm{X}_{\sigma(\m{X})}$ of $\m{X}$ in $\oK _n$. By Theorem \ref{thm:RoK}, construct a sequence $(\hm{Y}_i)_{1\leq i \leq n}$ of elements of $\oK_n$ such that \[ \hm{Y}_1 \arrows{(\hm{Y})}{\hm{X}_1 }{k} \] and \[ \hm{Y}_{i+1} \arrows{(\hm{Y}_i)}{\hm{X}_i}{k}. \]
 
Set $\hm{Z}=\hm{Y}_{\sigma(\m{X})}$, and let $\m{Z}$ denote the reduct of $\hm{Z}$ in $\oC _n$. We claim that $\m{Z}$ is as required. Recall that $\binom{\m{Z}}{\m{X}}$ denotes the set of induced isomorphic copies of $\m{X}$ in $\m{Z}$. Let \[ \alpha : \funct{\binom{\m{Z}}{\m{X}}}{[k]}. \]

It induces a map with values in $[k]$ on each of the sets $\binom{\hm{Z}}{\hm{X}_i}$, with $i \in [\sigma(\m{X})]$. By construction of $\hm{Z}=\hm{Y}_{\sigma(\m{X})}$, we can find a copy $\mc{Y}_{\sigma(\m{X})-1}$ of $\hm{Y}_{\sigma(\m{X})-1}$ in $\hm{Z}$ where all copies of $\hm{X}_{\sigma(\m{X})}$ have same $\alpha$-value. Then, by construction of $\hm{Y}_{\sigma(\m{X})-1}$, find a copy $\mc{Y}_{\sigma(\m{X})-2}$ of $\hm{Y}_{\sigma(\m{X})-2}$ in $\mc{Y}_{\sigma(\m{X})-1}$ where all copies of $\hm{X}_{\sigma(\m{X})-1}$ have same $\alpha$-value. Note that in $\mc{Y}_{\sigma(\m{X})-2}$, all copies of $\hm{X}_{\sigma(\m{X})}$ also have same $\alpha$-value. Repeating this process $\sigma(\m{X})$ times, we end up with a copy $\mc{Y}$ of $\hm{Y}$ in $\hm{Z}$ where, for every $i \in [\sigma(\m{X})]$, all copies of $\hm{X}_i$ have same $\alpha$-value. Thus, on $\binom{\m{Y}}{\m{X}}$, the map $\alpha$ takes no more than $\sigma(\m{X})$ values. $\qed$

\subsection{Proof of $t_{\oC _n}(\m{X}) \geq \sigma(\m{X})$}

\label{subsection:t>sigma}

Let $\hm{Y}$ be any disjoint union of all extensions of $\m{X}$ in $\oK _n$, together with an ordered copy of $K_n$ (where all vertices necessarily receive different $\lambda^{\hm{Y}}$-colors). 

\begin{claimm}
\label{claim:Yeq}
Let $f$ be a permutation of $[n]$. Then every extension of $\m{X}$ in $\oK _n$ embeds in $(\hat{Y}, E^{\hm{Y}}, <^{\hm{Y}}, f \circ \lambda^{\hm{Y}})$. 
\end{claimm}

\begin{proof}
Let $\hm{X}$ be an extension of $\m{X}$ in $\oK _n$. Consider $(\hat{X}, E^{\hm{X}}, <^{\hm{X}}, f^{-1} \circ \lambda^{\hm{X}})$. It is an extension of $\m{X}$, so it embeds in $\hm{Y}$. Under the new coloring $\lambda$, it becomes isomorphic to $\hm{X}$. $\qedhere$ \end{proof}

For $i \in [n]$, let $\hm{K}_1 ^i$ denote the one-point ordered graph with color $i$. By Theorem \ref{thm:RoK}, construct a sequence $(\hm{Y}_i)_{1\leq i \leq n}$ of elements of $\oK_n$ such that \[ \hm{Y}_1 \arrows{(\hm{Y})}{\hm{K}_1 ^1}{n^2} \] and whenever $2\leq i < n$, \[ \hm{Y}_{i+1} \arrows{(\hm{Y}_i)}{\hm{K}_1 ^{i+1}}{n^2}. \] 

Set $\hm{Z}=\hm{Y}_{n}$.

\begin{claimm}
\label{lem:oKeq}
For every coloring $\lambda : \funct{(Z, E^{\m{Z}})}{[n]}$, there is a copy $\mc{Y}$ of $\hm{Y}$ in $\hm{Z}$ where $\lambda = f \circ \lambda^{\mc{Y}}$ for some permutation $f$ of $[n]$.   
\end{claimm}

\begin{proof}
Let $\lambda : \funct{(Z, E^{\hm{Z}})}{[n]}$ be a coloring. Define $\mu : \funct{Z}{[n]\times[n]}$ by $\mu(z)=(\lambda^{\hm{Z}}, \lambda(z))$. By construction of $\hm{Z}=\hm{Y}_n$, we can find a copy $\mc{Y}_{n-1}$ of $\hm{Y}_{n-1}$ in $\hm{Z}$ where all vertices with $\lambda^{\mc{Y}_{n-1}}$-color $n$ have same $\mu$-color, and thus same $\lambda$-color. Then, by construction of $\hm{Y}_{n-1}$, find a copy $\mc{Y}_{n-2}$ of $\hm{Y}_{n-2}$ in $\mc{Y}_{n-1}$ where all vertices with $\lambda^{\mc{Y}_{n-2}}$-color $(n-1)$ have same $\mu$-color. Note that in $\mc{Y}_{n-2}$, vertices with $\lambda^{\mc{Y}_{n-2}}$-color $n$ also have same $\lambda$-color. Repeating this process $n$ times, we end up with a copy $\mc{Y}$ of $\hm{Y}$ in $\mc{Y}_{n}$ where all vertices with same $\lambda^{\mc{Y}}$-color have same $\lambda$-color. Note that because of the presence of the complete $K_n$ in $\hm{Y}$, no two vertices with different $\lambda^{\mc{Y}}$-color can have same $\lambda$-color, because $\lambda$ would then give same color to two adjacent vertices. It follows that on $\tilde{Y}$, $\lambda = f \circ \lambda^{\mc{Y}}$ for some permutation $f$ of $[n]$. $\qedhere$ \end{proof}

Let $\m{Z}$ denote the reduct of $\hm{Z}$ in $\oC _n$. It should be clear from the two previous claims that: 

\begin{claimm}
\label{claim:extoK}
For every extension $(\m{Z}, \lambda)$ of $\m{Z}$ in $\oK _n$, every extension of $\m{X}$ in $\oK _n$ embeds in $(\m{Z}, \lambda)$. 
\end{claimm}

Therefore: 

\begin{claimm}
\label{claim:RdoC}
For every $\m{T} \in \oC _n$, there is $\alpha : \funct{\binom{\m{T}}{\m{X}}}{[\sigma(\m{X})]}$ such that for every copy $\mc{Z}$ of $\m{Z}$ in $\m{T}$, $\alpha$ has range $[\sigma(\m{X})]$ on $\binom{\mc{Z}}{\m{X}}$.   
\end{claimm}

\begin{proof}
Let $\m{T} \in \oC _n$, and consider $\hm{T}$ any extension of $\m{T}$ in $\oK _n$. Enumerate the extensions $\hm{X}_1 ,\ldots, \hm{X}_{\sigma(\m{X})}$ of $\m{X}$ in $\oK _n$. On every copy $\mc{X}$ of $\m{X}$ in $\m{T}$, $\lambda^{\hm{T}}$ induces an extension of $\m{X}$ in $\oK_n$, isomorphic to $\hm{X}_i$ for some $i \in [\sigma(\m{X})]$. Denote this $i$ by $\alpha(\mc{X})$. Now, let $\mc{Z}$ be a copy of $\m{Z}$ in $\m{T}$. Then $(\mc{Z}, \restrict{\lambda^{\hm{T}}}{\tilde{Z}})$ is an extension of $\m{Z}$ in $\oK _n$ and by Claim \ref{claim:extoK}, $\alpha$ has range $[\sigma(\m{X})]$ on $\binom{\mc{Z}}{\m{X}}$. $\qedhere$ \end{proof}

The inequality $t_{\oC _n}(\m{X}) \geq \sigma(\m{X})$ follows. 

\subsection{Computation of $t_{\oC _n}(\m{X})$ in elementary cases}

We conclude this section with computations providing some values for $t_{\oC _n}(\m{X})$. The easiest case is obtained when $\m{X}$ is the complete ordered graph $K_m$, with $m\leq n$. It is clear that then, the number of extensions in $\oK _n$ is then equal to $\binom{n}{m} m!$. More generally, when $\m{X}$ is complete $m$-partite, with $m\leq n$ and each part being of size $l$, the number of extensions in $\oK _n$ is then equal to $\binom{n}{m} (lm)!$. Note that when $n>1$, none of those values is ever equal to $1$.

\section{Proof of Theorem \ref{thm:RdC}}

\label{section:RdC}

Let $\m{X} \in \C _n$. We follow the same strategy as in the previous section, and show successively that $t_{\C _n}(\m{X}) \leq \tau(\m{X})$ and that $t_{\C _n}(\m{X}) \geq \tau(\m{X})$. Recall that $\oM _n$ denotes the class of all monotone elements $\m{X}$ of $\oK _n$, ie where the coloring function $\lambda^{\m{X}}$ (seen as a map from $(X,<^{\m{X}})$ to $[n]$) is increasing. Our main ingredient here is the following consequence of Theorem \ref{thm:RoK}:  

\begin{lemma}
\label{thm:RoM}
The class $\oM _n$ has the Ramsey property.
\end{lemma}

\subsection{Proof of Lemma \ref{thm:RoM}}

We derive Lemma \ref{thm:RoM} from Theorem \ref{thm:RoK}. Let $\hm{X}, \hm{Y} \in \oM _n$ and $k \in \N$. By enriching $\hm{Y}$ if necessary, we may assume that $\lambda^{\hm{Y}}$ takes all values in $[n]$. By Theorem \ref{thm:RoK}, find $\hm{Z} \in \oK _n$ such that \[ \hm{Z} \arrows{(\hm{Y})}{\hm{X}}{k}.\] 

Of course, $\hm{Z}$ may not be monotone but it is easy to see that there is a unique linear ordering $<$ that coincides with $<^{\hm{Z}}$ on all $\lambda^{\hm{Z}}$-preimages and makes the new structure $\hm{Z} ^* := (\hat{Z}, E^{\hm{Z}}, <, \lambda^{\hm{Z}})$ monotone. 

\begin{claimm}
$\hm{Z}^* \arrows{(\hm{Y})}{\hm{X}}{k}$.
\end{claimm} 

\begin{proof}
Let \[ \alpha : \funct{\binom{\hm{Z}^*}{\hm{X}}}{[k]}. \]

Observe that monotone substructures in $\hm{Z}$ induce monotone substructures in $\hm{Z} ^*$ and vice-versa. In particular, $\hm{Z}$ and $\hm{Z} ^*$ have exactly the same copies of $\hm{X}$ and $\hm{Y}$. It follows that $\alpha$ can be seen as a map from $\binom{\hm{Z}}{\hm{X}}$ to $[k]$. By construction of $\hm{Z}$, there is a copy $\mc{Y}$ of $\hm{Y}$ in $\hm{Z}$ (hence in $\hm{Z}^*$) such that $\binom{\mc{Y}}{\hm{X}}$ is $\alpha$-monochromatic. $\qedhere$ \end{proof}

\subsection{Proof of $t_{\C _n}(\m{X}) \leq \tau(\m{X})$}

Proceed as in section \ref{subsection:t<sigma}, working with $\C _n$ instead of $\oC _n$ and $\oM _n$ instead of $\oK _n$. $\qed$ 

\subsection{Proof of $t_{\C _n}(\m{X}) \geq \tau(\m{X})$}

The proof resembles the proof performed in \ref{subsection:t>sigma}. Say that two linear orderings $<$ and $<'$ are \emph{opposite} on a given set when $x<y$ iff $y<'x$ on that set. Let $\hm{Y}$ be any disjoint union of all extensions of $\m{X}$ in $\oM _n$, together with an ordered copy of $K_n$ (where all vertices necessarily receive different $\lambda^{\hm{Y}}$-colors).

\begin{claimm}
\label{claim:Yeq<}
Let $<$ be a linear ordering on $\hat{Y}$ such that on every $\lambda^{\hm{Y}}$-preimage, $<$ and $<^{\hm{Y}}$ coincide or are opposite. Let $f$ be a permutation of $[n]$ and set $\lambda = f \circ \lambda^{\hm{Y}}$. Assume that $(\hat{Y}, E^{\hm{Y}}, <, \lambda) \in \oM _n$. Then every extension of $\m{X}$ in $\oM _n$ embeds in $(\hat{Y}, E^{\hm{Y}}, <, \lambda)$. 
\end{claimm}

\begin{proof}

We show that $(\hat{Y}, E^{\hm{Y}}, <, \lambda)$ contains at least as many isomorphism types of extension of $\m{X}$ in $\oM _n$ as $\hm{Y}$ does. Observe first that every extension of $\m{X}$ in $\hm{Y}$ is still an extension of $\m{X}$ in $(\hat{Y}, E^{\hm{Y}}, <, \lambda)$. So it suffices to show that if two subsets $X_1$ and $X_2$ of $\hat{Y}$ support non isomorphic substructures in $\hm{Y}$, then they also support non isomorphic substructures in $(\hat{Y}, E^{\hm{Y}}, <, \lambda)$. Or equivalently, if $X_1$ and $X_2$ support isomorphic substructures in $(\hat{Y}, E^{\hm{Y}}, <, \lambda)$, then they also do in $\hm{Y}$. 

So let $X_1, X_2$ be subsets of $\hat{Y}$ and $f:\funct{X_1}{X_2}$ witnessing that $X_1$ and $X_2$ support isomorphic substructures of $(\hat{Y}, E^{\hm{Y}}, <, \lambda)$. We show that $f$ is also an isomorphism between the corresponding substructures of $\hm{Y}$. First, it is clear that $f$ preserves the graph structure $E^{\hm{Y}}$ as well as the coloring $\lambda^{\hm{Y}}$. So it remains to show that it preserves $<^{\hm{Y}}$. Let $x <^{\hm{Y}} y \in X_1$. Assume first that $x$ and $y$ belong to the same $\lambda^{\hm{Y}}$-preimage, call it $A$. If $<$ and $<^{\hm{Y}}$ coincide on $A$, then $x < y$. Since $f$ preserves $<$ and $\lambda^{\hm{Y}}$, we have $f(x)<f(y)$ in $A$. So $f(x) <^{\hm{Y}} f(y)$. Similarly, if $<$ and $<^{\hm{Y}}$ are opposite on $A$, then $y < x$. So $f(y)<f(x)$ in $A$. Hence, $f(x) <^{\hm{Y}} f(y)$ and we are done in the case where $\lambda^{\hm{Y}}(x) = \lambda^{\hm{Y}}(y)$. Now, assume that $\lambda^{\hm{Y}}(x)\neq\lambda^{\hm{Y}}(y)$. Then $x' <^{\hm{Y}} y'$ whenever $\lambda^{\hm{Y}}(x')=\lambda^{\hm{Y}}(x)$ and $\lambda^{\hm{Y}}(y')=\lambda^{\hm{Y}}(y)$. Since $f$ preserves $\lambda^{\hm{Y}}$, we may take $f(x)=x'$, $f(y)=y'$, and $f(x) <^{\hm{Y}} f(y)$ follows. $\qedhere$ \end{proof}

For $i \in [n]$, denote by $\hm{K}_2 ^i$ the element of $\oM _n$ consisting of two non adjacent vertices, both with color $i$. By Lemma \ref{thm:RoM}, construct a sequence $(\hm{Y}_i)_{1\leq i \leq n}$ of elements of $\oM_n$ such that \[ \hm{Y}_1 \arrows{(\hm{Y})}{\hm{K}_2 ^1}{2} \] and whenever $2\leq i < n$, \[ \hm{Y}_{i+1} \arrows{(\hm{Y}_i)}{\hm{K}_2 ^{i+1}}{2}. \] 

Set $\hm{Z}=\hm{Y}_{n}$. 

\begin{claimm}
\label{claim:<}
For every linear ordering $<$ with respect to which $\lambda^{\hm{Z}}$-preimages are $<$-convex, there is a copy $\mc{Y}$ of $\hm{Y}$ in $\hm{Z}$ such that on $\lambda^{\mc{Y}}$-preimages, $<$ and $<^{\mc{Y}}$ coincide or are opposite. 
\end{claimm}

\begin{proof}

Let $<$ be a linear ordering with respect to which $\lambda^{\hm{Z}}$-preimages are $<$-convex. Let $\beta$ be a coloring of the non adjacent pairs of vertices, assigning value $0$ to a pair if $<$ and $<^{\hm{Z}}$ are opposite on that pair, and value $1$ otherwise. By construction of $\hm{Z}=\hm{Y}_n$, find a copy $\mc{Y}_{n-1}$ of $\hm{Y}_{n-1}$ in $\hm{Z}$ where all non adjacent pairs with $\lambda^{\mc{Y}_{n-1}}$-color $n$ have same $\beta$-color. Then, by construction of $\hm{Y}_{n-1}$, find a copy $\mc{Y}_{n-2}$ of $\hm{Y}_{n-2}$ in $\mc{Y}_{n-1}$ where all non adjacent pairs with $\lambda^{\mc{Y}_{n-2}}$-color $(n-1)$ have same $\beta$-color. Note that in $\mc{Y}_{n-2}$, non adjacent pairs with $\lambda^{\mc{Y}_{n-2}}$-color $n$ also have same $\beta$-color. Repeating this process $n$ times, we end up with a copy $\mc{Y}$ of $\hm{Y}$ in $\hm{Y}_{n}$ where all non adjacent pairs with same $\lambda^{\mc{Y}}$-color have same $\beta$-color. Then, on $\lambda^{\mc{Y}}$-preimages, $<$ and $<^{\mc{Y}}$ coincide or are opposite. $\qedhere$ \end{proof}

Now, by Lemma \ref{thm:RoM}, construct a sequence $(\hm{Z}_i)_{1\leq i \leq n}$ of elements of $\oM_n$ such that \[ \hm{Z}_1 \arrows{(\hm{Z})}{\hm{K}_1 ^1}{n^2} \] and whenever $2\leq i < n$, \[ \hm{Z}_{i+1} \arrows{(\hm{Z}_i)}{\hm{K}_1 ^{i+1}}{n^2}. \] 

Set $\hm{T}=\hm{Z}_{n}$. 

\begin{claimm}
\label{claim:oMeq} 
For every coloring $\lambda : \funct{(\hat{T}, E^{\hm{T}})}{[n]}$, there is a permutation $f$ of $[n]$ and a copy $\mc{Z}$ of $\hm{Z}$ in $\hm{T}$ where $\lambda = f \circ \lambda^{\mc{Z}}$.  
\end{claimm}

\begin{proof} Same proof as for Claim \ref{lem:oKeq}. $\qedhere$ \end{proof}

\begin{claimm}
Let $<$ be a linear ordering on $\hat{T}$ and $\lambda : \funct{(\hat{T}, E^{\hm{T}})}{[n]}$ be a coloring so that $(\hat{T}, E^{\hm{T}}, <, \lambda) \in \oM _n$. Then there is a copy $\mc{Y}$ of $\hm{Y}$ in $\hm{T}$ where $<$ and $<^{\mc{Y}}$ coincide or are opposite on $\lambda^{\mc{Y}}$-preimages, and where $\lambda = f \circ \lambda^{\mc{Y}}$ for some permutation $f$ of $[n]$. 
\end{claimm}

\begin{proof}
By Claim \ref{claim:oMeq}, we can find a permutation $f$ of $[n]$ and a copy $\mc{Z}$ of $\hm{Z}$ in $\hm{T}$ where $\lambda = f \circ \lambda^{\mc{Z}}$. Since on $\tilde{Z}$, $\lambda$-preimages are $<$-convex, so are $\lambda^{\mc{Z}}$-preimages. By Claim \ref{claim:<}, there then there is a copy $\mc{Y}$ of $\hm{Y}$ in $\mc{Z}$ where $<$ and $<^{\mc{Y}}$ coincide or are opposite on $\lambda^{\mc{Y}}$-preimages. Note that on $\tilde{Y}$, $\lambda = f \circ \lambda^{\mc{Y}}$, so $\mc{Y}$ is as required. $\qedhere$ \end{proof}

Let $\m{T}$ denote the reduct of $\hm{T}$ in $\C _n$. Together with Claim \ref{claim:Yeq<}, the previous claim directly implies: 

\begin{claimm}
Let $(\m{T}, <, \lambda)$ be an extension of $\m{T}$ in $\oM _n$. Then every extension of $\m{X}$ in $\oM _n$ embeds in $(\m{T}, <, \lambda)$. 
\end{claimm}

Therefore, as in section \ref{subsection:t<sigma}, we obtain: 

\begin{claimm}
For every $\m{U} \in \C _n$, there is $\alpha : \funct{\binom{\m{U}}{\m{X}}}{[\tau(\m{X})]}$ such that for every copy $\mc{T}$ of $\m{T}$ in $\m{U}$, $\alpha$ has range $[\tau(\m{X})]$ on $\binom{\mc{T}}{\m{X}}$.   
\end{claimm}

\begin{proof} Same as proof of Claim \ref{claim:RdoC}. $\qedhere$ \end{proof}

The inequality $t_{\C _n}(\m{X}) \geq \tau(\m{X})$ follows. 

\subsection{Computation of $t_{\C _n}(\m{X})$ in elementary cases}

As previously, we conclude this section with computations providing some values for $t_{\C _n}(\m{X})$. When $\m{X}$ is the complete ordered graph $K_m$, with $m\leq n$, we have $t_{\C _n}(\m{X}) = \binom{n}{m}$. Note that the same result holds when $\m{X}$ is complete $m$-partite, with $m\leq n$ and each part being of any size $l$. Therefore, $t_{\C _n}(\m{X})=1$ whenever $m=n$ in any of those cases.

\end{document}